\documentclass[11pt]{amsart}

\usepackage{amsmath,amssymb,amsthm, mathrsfs}
\usepackage{tikz}
\usepackage{circuitikz}
\usepackage{graphicx,multicol}
\usepackage{comment,relsize,braket}
\usepackage{enumerate,multicol}
\usepackage{blkarray}
\usepackage[colorlinks=true]{hyperref}
\usepackage{thmtools,thm-restate}
\definecolor{red}{rgb}{1.00,0.00,0.00}
\usepackage[top=1in, bottom=1in, left=1in, right=1in]{geometry}
\usepackage{xcolor,mathtools,cleveref}

\usepackage{tabularray}

\def\ker{\operatorname{Ker}}
\def\im{\operatorname{Im}}
\def\rk#1{\hbox{\rm rank}\,(#1)}

\numberwithin{equation}{section}

\newtheorem{theorem}{Theorem}[section]
\newtheorem{lemma}[theorem]{Lemma}
\newtheorem{corollary}[theorem]{Corollary}
\newtheorem{proposition}[theorem]{Proposition}
\newtheorem{example}[theorem]{Example}

\newtheorem{remark}[theorem]{Remark}
\newtheorem{definition}[theorem]{Definition}

\newtheorem*{notation}{Notation}
\newtheorem*{acknowledgment*}{Acknowledgment}

\def\N{\mathbb{N}}

\def\Z{\mathbb{Z}}

\newcommand{\D}{\Delta}
\renewcommand{\d}{\delta}

\renewcommand{\l}{\lambda}

\newcommand{\bs}{\setminus}
\newcommand{\q}{\quad}

\newcommand{\la}{\langle}
\newcommand{\ra}{\rangle}

\renewcommand{\b}{\beta_{1,\l}}
\newcommand{\I}{\mathcal I}
\newcommand*{\SM}[1]{S^e(#1)}
\newcommand*{\IM}[1]{I^e(#1)}

\newcommand*{\cIM}[1]{\mathcal I^e(#1)}
\newcommand*{\cSM}[1]{\mathcal S^e(#1)}
\newcommand*{\JM}[1]{\mathcal J^e(#1)}

\renewcommand{\ll}{\left\lfloor}
\newcommand{\rr}{\right\rfloor}
\newcommand{\lc}{\left\lceil}
\newcommand{\rc}{\right\rceil}
\def\wid{\operatorname{wd}}
\newcommand{\fsy}{f}

\begin{document}
\title{numerical semigroups of Sally Type}

\author{Saipriya Dubey}
\address{Chennai Mathematical Institute, India}
\email{saipriyad@cmi.ac.in}
\author{Kriti Goel}
\address{Basque Center for Applied Mathematics, Alameda de Mazarredo 14, Bilbao, Bizkaia, 48009 Spain}
\email{kritigoel.maths@gmail.com}
\author{N\.{i}l \c{S}ah\.{i}n}
\address{Department of Industrial Engineering, Bilkent University, Ankara, 06800 Turkey}
\email{nilsahin@bilkent.edu.tr}
\author{Srishti Singh}
\address{Mathematics Department, University of Missouri, Columbia, MO 65211, USA.}
\email{spkdq@umsystem.edu}
\author{Hema Srinivasan}
\address{Mathematics Department, University of Missouri, Columbia, MO 65211, USA.}
\email{srinivasanh@missouri.edu}

\thanks{The workshop was hosted and financially supported by Banff International Research Station at its CMO Oaxaca location. Additional financial support was provided by the NSF grant DMS–2433082.}

\dedicatory{Dedicated to Judith Sally for her lifetime contributions to Algebra as well as mentoring many mathematicians.}
  
\begin{abstract}
     Judith Sally proved in 1980 that the associated graded ring of one-dimensional Gorenstein local rings of multiplicity $e$ and embedding dimension $e-2$ are Cohen-Macaulay. She showed that the defining ideal of the associated graded ring of such rings is generated by ${e-2 \choose 2}$ elements. Numerical semigroup rings are a big class of one-dimensional Cohen-Macaulay rings. In 2014, Herzog and Stamate proved that the numerical semigroup $\la e,e+1,e+4,\ldots,2e-1 \ra$ defines a Gorenstein semigroup ring satisfying Sally's conditions above and such semigroups are called Gorenstein Sally Semigroups. We call a numerical semigroup as {\it Sally type} if its multiplicity is one more than its width. In this paper, we study Sally type numerical semigroups of the form $\SM{m,n} = \la \{ e,e+1,\ldots,2e-1\} \bs \{e+m,e+n\} \ra,$ for some $2 \leq m <n \leq e-2$.  We give a formula for its Frobenius number along with a necessary and sufficient condition for it to be Gorenstein. We compute the minimal number of generators for the defining ideal of $k[\SM{m,n}]$. Additionally, we present an algorithm and a GAP code used in applying Hochster's combinatorial formula to compute the first Betti number of $k[\SM{m,n}]$.
\end{abstract} 

\maketitle
\section{Introduction}

Let $S=\langle s_1,s_2,\ldots,s_g \rangle=\left\{\sum\limits_{i=1}^{g}u_is_i \mid u_i\in \mathbb{N}\right\}$  be a numerical semigroup, where $s_1<s_2<\cdots<s_g$ is a sequence of positive integers with $\gcd(s_1,s_2,\hdots,s_g)=1$. Let $k$ be a field, and $A = k[X_1, X_2, \dots, X_g] $ the polynomial ring in $g$ variables. Consider the $k$-algebra homomorphism $ \phi: A \to k[S]$  defined by
$\phi(X_i) = t^{s_i}$, where  $k[S]=k[t^{s_1},\hdots,t^{s_g}]$ is the semigroup ring of $S$. The kernel of this homomorphism, denoted by $I_S,$ is the defining ideal of the affine curve $C_S$ with parametrization $$X_1 = t^{s_1} , X_2=t^{s_2} , \ldots , X_g =t^{s_g}$$ corresponding to $S$. When $s_1,s_2,\ldots, s_g$ is a minimal set of generators for $S$, the difference $\wid(S) = s_g-s_1$ is the width of $S$, while $s_1$ is its multiplicity and $g$ is its embedding dimension. Let $R_S = k[[t^{s_1} , \ldots, t^{s_g} ]]$ be the local ring with the maximal ideal $\mathfrak{m} = \langle t^{s_1}, \ldots,t^{s_g} \rangle$. Then $gr_{\mathfrak{m}}(R_S )=\bigoplus_{i \ge 0} 
\mathfrak{m}_i/\mathfrak{m}_{i+1}\cong A/I^*_S$ is the associated graded ring where $I^*_S=\langle f^* \mid f\in I_S \rangle$ with $f^*$ denoting the initial form of $f$.

It is known that the minimal number of generators of $I_S$ for $g\geq 4$ is arbitrarily large; see \cite{Bresinsky}. Hence, it is natural to ask, for some specific classes of numerical semigroups $S$, ``is there a bound for $\mu(I_S)$?" or perhaps there is a good bound with fixing another invariant besides the embedding dimension. There is much work in the literature determining the generators and the first Betti number of the defining ideals of different classes of monomial curves; see \cite{Bresinsky,GimSenSri,Herzog,HerzogStamate,Stam,Valla}. The examples in \cite{Bresinsky} show that for any embedding dimension $g$, and any multiplicity $e$, $\mu(I_S)$ is unbounded. In this paper, we study the classes fixing both multiplicity and the width. 

J. Sally \cite{Sally} established the Cohen-Macaulayness of the associated graded ring of local Gorenstein rings whose difference between the embedding dimension and dimension is $e-3$ and explicitly wrote down the generators of the defining ideals. When the dimension is $1$, that is, if the embedding dimension equals $e-2$,  these are the defining ideals of a symmetric  numerical semigroup which are termed Sally Semigroups \cite{HerzogStamate} by  Herzog and Stamate. They showed that $S=\langle e,e+1,e+4,\hdots,2e-1 \rangle$ is a symmetric Sally semigroup and the number of minimal generators $\mu(I^\ast_S)$ of $I^\ast_S$ satisfies \[\mu(I^*_S)={e-2 \choose 2}\leq {\wid(S)+1 \choose 2} \] and conjectured that $\mu(I_S)\leq {\wid(S)+1 \choose 2}$ and in 2024 Caviglia et al. \cite{CMS} generalized this conjecture to $\beta_t\leq t{\wid(S)+1 \choose t+1}$, where $\beta_t$ is the $t^{\text{th}}$ Betti number of $I_S$.  Motivated by this, we say a numerical semigroup is of Sally type if its multiplicity is one more than its width.

Consider the following Sally type numerical semigroups with multiplicity $e$ and embedding dimension $e,$ $e-1,$ and $e-2$ respectively. 
 \begin{align*} 
 S^e &= \langle \{e+i\ |\  0\le i\le e-1\} \rangle, \ \SM{m}= \langle \{e+i\ |\  0\le i\le e-1, i \neq m\} \rangle, \\
\text{ and } \SM{m,n} &= \langle \{e+i\ |\  0\le i\le e-1, i \neq m, n\} \rangle 
 \end{align*}
where $1 \leq m<n\leq e-2.$ Observe that $S^e$ is an arithmetic sequence and it is the interval completion of $\SM{m}$ and $\SM{m,n}$. A minimal set of generators for $S^e$ is obtained by Patil in \cite{Patil1993239} and the minimal
free resolution and a formula for the Betti numbers is obtained by Gimenez, Sengupta and Srinivasan  in \cite{GimSenSri,Patil1993239}. Note that the Sally type numerical semigroup $\SM{2,3}$ is actually a Sally semigroup mentioned above (see \cite{HerzogStamate,Sally}).

In this paper, we compute the minimal generating set for the defining ideal $\IM{m,n}$ of the semigroup ring $k[\SM{m,n}]$. We start by describing the Sally type numerical semigroups of the form $\SM{m,n}$ in \Cref{section:sallytypesemigroups}. To obtain a minimal generating set of $\IM{m,n}$, we first compute the minimal number of generators combinatorially using Hochster's Formula \cite[Theorem 5.5.1]{BrunsHerzog}.  This is computationally intense, and we move some of the repetitive computations to the Appendix.  The background needed to apply Hochster's formula for our situation is given in \Cref{section:background}. In \Cref{Section:main}, we prove our main result:

\begin{restatable*}{theorem}{maintheorem}
\label{thm:maintheorem}
     Let $\SM{m,n}$ be a Sally type numerical semigroup where $2 \leq m < n \leq e-2$. Then the number of minimal generators of $\IM{m,n}$ is \begin{equation}\label{eq:maintheorem}
         \mu(\IM{m,n} ) = \begin{cases}
         {e-2 \choose 2}, & (m,n) \in \{(2,4),(2,5),(3,4)\} \\
         {e-2 \choose 2} -1, & (m=2) \text{ and } (n=3 \text{ or } n \geq 6) \\
             {e-2 \choose 2} -2, &  \text{ otherwise. }
         \end{cases}
     \end{equation}
\end{restatable*}
Note that it easily follows that the above numerical semigroups of  Sally type satisfy Herzog and Stamate's width conjecture mentioned above (see Corollary \ref{widthconj}).  Finally, using this theorem, a minimal generating set for $\IM{m,n}$ is given in \Cref{section:minimalgenerators}. We conclude with an algorithm and a GAP code based on the proof of our main theorem in \Cref{section:algorithm}.

\section{Numerical Semigroups of Sally Type}\label{section:sallytypesemigroups}

\begin{definition}
    A Sally type numerical semigroup $\SM{m,n}$ is a numerical semigroup minimally generated by the set $\cSM{m,n} = \{e,e+1,\ldots,2e-1\} \bs \{e+m,e+n\}$ for some $1 \leq m <n \leq e-2$.
\end{definition}

In this article, we only consider the case where $m \geq 2.$ 

\begin{notation}
    For ease of reading, we introduce the following notation: For $x<y$ in $\N$, $[x,y]$ is the set of integers $\{x,x+1,\ldots,y\}$. Let $\JM{m,n} = [0,e-1] \bs \{m,n\}$. 
\end{notation}
%three \SM in the following paragraph
Let $k$ be a field and $t$ be an indeterminate. The semigroup ring associated to $\SM{m,n}$ is given by $k[\SM{m,n}] = k[t^{e+i} \mid i \in \JM{m,n}].$ Let $A^e(m,n)=k[X_i: i \in \JM{m,n}]$ be a polynomial ring in $e-2$ variables. Consider the $k$-algebra homomorphism $ \phi: A^e(m,n) \to k[t]$  defined by
$\phi(X_i) = t^{e+i}$. Then $k[\SM{m,n}] \cong A^e(m,n) / \IM{m,n}$, where  $\IM{m,n} = \ker \phi$ is the defining ideal of $\SM{m,n}$. 

In order to understand the structure of a semigroup, the first step is to compute its Frobenius number; the greatest integer not in $\SM{m,n}.$ 

\begin{proposition}\label{prop:Frob of semigroup}
For $2 \leq m < n \leq e-2$, let $F(\SM{m,n})$ denote the Frobenius number of $\SM{m,n}.$ Then
\[ F(\SM{m,n}) = \begin{cases}
        2e+3, & (m,n) = (2,3), \\
        e + n, & \text{otherwise.}
    \end{cases}\]
\end{proposition}
\begin{proof}
 In order to prove the result, it is enough to check that $F(\SM{m,n}) \notin \SM{m,n}$ and for all $s$ such that $F(\SM{m,n}) < s \leq F(\SM{m,n})+e,$ we have $s \in \SM{m,n}.$ While it is easy to see that $F(\SM{m,n}) \notin \SM{m,n},$ we check the latter condition in each case separately.
 
 Let $(m,n)=(2,3).$ Then $\SM{2,3} = \la e,e+1,e+4,\ldots,2e-1 \ra.$ For any $s$ such that $2e+4 \leq s \leq 3e-1,$ we have $s= e + (e+j) \in \SM{2,3}$ where $4 \leq j \leq e-1,$ and for $s = 3e+x$ where $0 \leq x \leq 3,$ we have $s = (3-x) \cdot e + x \cdot (e+1) \in \SM{2,3}.$

Otherwise, if $(m,n) \neq (2,3),$ we show that $$e+n+1,e+n+2,\ldots,e+n+e\in \SM{m,n}.$$ 
Clearly, $e+n+1,e+n+2,\ldots,2e-1\in \SM{m,n}.$ It remains to show that $2e+\alpha \in \SM{m,n}$ where $0 \leq \alpha \leq n.$ Note that for $\alpha \in [0,n] \setminus \{m,n\}$, we have $2e+\alpha =e+(e+\alpha) \in \SM{m,n}.$ Also, $2e+m = (e+1) + (e+m-1) \in \SM{m,n}.$ Finally for $\alpha = n,$  we have $2e+n =e+(m-1)+[e+(n-m+1)] \in \SM{m,n}$ except for the cases when $n-m+1=m$ and $n-m+1=n.$ As $m \geq 2,$ the latter condition is not possible. When $n-m+1=m$ we have $2e+n = 2e+(2m-1)=(e+m-2)+(e+m+1) \in \SM{m,n}.$ 
\end{proof}

\begin{theorem} A Sally type numerical semigroup $\SM{m,n}$ is Gorenstein if and only if $(m,n) = (2,3).$
\end{theorem}

\begin{proof}
Note that $\SM{m,n}$ is Gorenstein if and only if $(F(\SM{m,n})+1)/2=g(\SM{m,n}),$ where $g(\SM{m,n})$ is the cardinality of gap set of $\SM{m,n}$ (see \cite[Corollary 4.5]{Rosales2009}).

When $(m,n)=(2,3),$ by \Cref{prop:Frob of semigroup} we get $F(\SM{2,3})=2e+3.$ Observe that $g(\SM{2,3})=e+2.$ Thus, 
$$\frac{F(\SM{2,3})+1}{2}=\frac{2e+3+1}{2}=e+2=g(\SM{2,3})$$
implying that $\SM{2,3}$ is Gorenstein.

When $(m,n) \neq (2,3),$ then by \Cref{prop:Frob of semigroup} we have $F(\SM{m,n})=e+n$ and observe that $g(\SM{m,n})=e+1.$ Then 
$$\frac{F(\SM{m,n})+1}{2}=\frac{e+n+1}{2}$$
is equal to the genus $g(\SM{2,3}) = e+1$ if and only if $n=e+1.$ Since this contradicts our choice of $n,$ we get that $\SM{m,n}$ is never Gorenstein in this case. Hence, the result holds.
\end{proof}

\begin{example} %\SM{ in this example once
    $S^7(2,3) = \la 7,8,11,12,13 \ra$ is a Sally type numerical semigroup with Frobenius number $17$, and is Gorenstein.
\end{example}

In order to obtain a minimal generating set for $\IM{m,n}$, we first compute the minimum number of generators, $\mu(\IM{m,n})$, using Hochster's Formula \cite[Theorem 5.5.1]{BrunsHerzog}. To understand the technique behind this theorem, we need some background.

\section{Background: Homology of Stanley-Reisner Rings}\label{section:background}
For $\l \in S$, define the simplicial complex $\Delta_\l$ on the vertex set $V=\{s_1,\ldots,s_g\}$ as follows: $F=\{s_{i_0},\ldots,s_{i_r}\}$ for $1 \leq i_0 < \cdots < i_r \leq g$ is a face of $\D_\l$ if and only if $\lambda-\sum\limits_{j=0}^{r}s_{i_j} \in S.$ A face $F$ is called an $r$-dimensional face if $|F| = r+1.$ The dimension of the simplicial complex $\D_\l$ is defined as $\dim(\D_\l) = \max\{\dim F \mid F \in \D_\l\}.$

\begin{example}\label{ex:simplicialcomplex}
    Let $S = \la 7,8,11,12,13 \ra$, and $\l = 19 \in S$. Then $$\D_\l = \{\varnothing, \{7\}, \{8\}, \{11\}, \{12\}, \{7,12\}, \{8,11\}\}.$$
\end{example}

Define the \textit{augmented oriented chain complex} of $\D_\l$ as follows \cite[Section 5.3]{BrunsHerzog}: Suppose $\D_\l$ has dimension $d-1$, and $F = \{s_{i_0},\ldots,s_{i_r} \} \in \D_\l$ is an $r$-dimensional face for some $-1 \leq r \leq d-1$. Then
\begin{align*} 
    \tilde{\mathcal C}(\D_\l): 0 \to C_{d-1} \xrightarrow{\delta_{d-1}} C_{d-2} \to \cdots \to C_0 \xrightarrow{\delta_0} C_{-1} \to 0 \\
    \text{ where } \qquad C_r = \bigoplus_{\substack{F \in \D_\l,\\ \dim F =r}} F, \qquad \delta_r F = \sum_{j=0}^r (-1)^j F_j,
\end{align*}
where $F_j = \{s_{i_0},\ldots, \widehat{s_{i_j}},\ldots,s_{i_r}\}.$ Note that for any zero-dimensional face $F=\{s\},$ $\d_0(F) = \varnothing.$
\begin{example}\label{ex:simplicialcomplex2}
    Let $S$ and $\l$ be as in $\Cref{ex:simplicialcomplex}.$ Then 
    \[\tilde{\mathcal C}(\D_\l): 0 \to \Z^2 \xrightarrow{\small \begin{bmatrix}
       -1 & 0 \\
       0 & -1 \\
       0 & 1 \\
       1 & 0
    \end{bmatrix}} \Z^4 \xrightarrow{\small \begin{bmatrix}
        1 & 1 & 1 & 1
    \end{bmatrix}} \Z \to 0.\]
\end{example}

Now we state our key prerequisite that gives a combinatorial interpretation of the fine Betti numbers of $\D_\l.$ 

\begin{theorem}\label{thm:hochster'sformula}{\rm [Hochster's Formula, \cite[Theorem 5.5.1]{BrunsHerzog}]}
    For $\l \in S$, $0 \leq i \leq d$, 
    \[\beta_{i, \lambda} = \dim_k H_{i-1} \left( \tilde{\mathcal C}(\D_\l) \otimes k \right) = \dim_k \left(\frac{\ker \delta_{i-1} \otimes k}{\im \delta_{i} \otimes k} \right),\]
    where $H_{\bullet}(\tilde{\mathcal C}(\D_\l) \otimes k)$ denotes the reduced simplicial homology of $\D_\l$ with coefficients in $k$. Then the $i$-th Betti number of $k[S],$ $\beta_i = \sum_{\l \in S} \beta_{i,\l},$ for all $1 \leq i \leq d.$ 
\end{theorem}

\begin{example}
    Let $S$ and $\l$ be as in \Cref{ex:simplicialcomplex}. We have the complex $\tilde{\mathcal C}$ from $\Cref{ex:simplicialcomplex2}$. It can be seen that $\ker \delta_0 \cong \Z^3$, and $\im \delta_1 \cong \Z^2$, so $\beta_{1,\l}=1$.
\end{example}

We end this section by stating a technical lemma about the complex $\tilde C(\D_\l)$, which will be used extensively in \Cref{Section:main}. 

\begin{lemma}\label{lem:ranks}
    Let $S$ be a numerical semigroup minimally generated by $g$ elements, and $\l \in S$. Let $\D_\l$ be the corresponding simplicial complex and $\tilde{\mathcal C}(\D_\l)$ the chain complex. Let $M$ denote the matrix representation of $\d_1$. Suppose $\D_\l$ has $\ell$ zero-dimensional faces.
    \begin{enumerate}[{\rm(1)}]
        \item Then $C_{-1}=\Z.$
        
        \item  Then $K_0  := \ker \delta_0 \cong \Z^{\ell-1}$. In particular, $\rk M \leq \ell-1 \leq g-1$.

        \item  Suppose $\{a,a_1\},\ldots, \{a,a_r\}$, $r < \ell$ are distinct $1$-dimensional faces of $\D_\l$. Then $\rk M \geq r$. Furthermore, if $\{b,c\}$ with $c \neq a, c \neq a_i$, $i=1,\ldots,r$ is another one-dimensional face of $\D_\l$, then $\rk M \geq r+1$.
        
        \item Suppose $\{a_1,b_1\},\ldots,\{a_r,b_r\}$ are $1$-dimensional faces of $\D_\l$ such that each $a_i$ is distinct from each $a_j$ and from each $b_l$. Then $\rk M \geq r$.   
    \end{enumerate}
\end{lemma}

\begin{proof}
\begin{enumerate}[(1)]
    \item Note that for any $\l \in S$, $\varnothing \in \D_\l.$ In fact, it is the only face of dimension $-1$ in $\D_\l$. So $C_{-1}=\Z$.

    \item Recall that $\d_0: C_0=\mathbb{Z}^\ell \rightarrow C_{-1}=\mathbb{Z}$ and for any zero-dimensional face $F$ of $\D_\l,$ $\d_0(F) = \varnothing$ implying that $\d_0$ is a non-zero map. Thus, $\rk{\d_0} =1$, $\im \d_0 \cong \mathbb{Z}$ and hence $K_0 \cong \mathbb{Z}^{\ell-1}.$

    \item If $\{a,a_1\},\ldots, \{a,a_r\}$, are distinct $1$-dimensional faces of $\D_\l$, then $M$ has the following submatrix
    {\small \[M' = \begin{blockarray}{cccccc}
    & \{a,a_1\} & \{a,a_2\} & \cdots & \{a,a_{r-1}\} & \{a,a_r\} \\
    \begin{block}{c[ccccc]}
    a & -1 & -1 & \cdots & -1 & -1 \\
    a_1 & 1 & 0 & \cdots & 0 & 0 \\
    a_2 & 0 & 1 & \cdots & 0 & 0 \\
    \vdots & \vdots & \vdots & \ddots & \vdots & \vdots \\
    a_{r-1} & 0 & 0 & \cdots & 1 & 0 \\
    a_r & 0 & 0 & \cdots & 0 & 1 \\
    \end{block} 
    \end{blockarray}\] }
    
    Since $\rk{M'}=r,$ we get $\rk M \geq r$.

    Now, for an additional one-dimensional face $\{b,c\}$ such that $c \neq a, c \neq a_i$, for any $i=1,\ldots,r$, consider a submatrix $M''$ of $M$ with columns corresponding to the one-dimensional faces $\{a,a_1\},\ldots, \{a,a_r\}, \{b,c\}.$ Since these columns are linearly independent, we get $\rk{M''} = r+1.$ Thus, in this case, $\rk M \geq r+1.$

    \item WLOG, assume $a_1<a_2<\cdots<a_r.$ In this case, $M$ has the following rank $r$ submatrix 
    {\small \[M' = 
    \begin{blockarray}{cccccc}
        &  \{a_1,b_1\} & \{a_2,b_2\} & \cdots & \{a_{r-1},b_{r-1}\} &  \{a_r,b_r\} \\
        \begin{block}{c[ccccc]}
            a_1 & -1 & 0 & \cdots & 0 & 0 \\
            a_2 & 0 & -1 & \cdots & 0 & 0 \\
            \vdots & \vdots & \vdots & \ddots & \vdots & \vdots \\
            a_{r-1} & 0 & 0 & \cdots & -1 & 0 \\
            a_r & 0 & 0 & \cdots & 0 & -1 \\ \cline{2-6} \\
             b_1 &  &  &  &  &  \\
             b_2 & & & & & \\
             \vdots & & & \ast & & \\
             b_{r-1} & & & & & \\
             b_r & & & & & \\
        \end{block}
    \end{blockarray}
    \]} 
    as irrespective of the block entries of ~$*$, the columns of the matrix $M'$ are linearly independent. \qedhere
\end{enumerate}
\end{proof}

\section{Main Theorem}\label{Section:main}

Throughout this section, we will use the following notation: $\tilde C(\D_\l)$ is the complex associated with $\l \in \SM{m,n}$, and $M$ denotes the matrix representation of $\d_1.$

\begin{remark}\label{rem:ideaofproof}
In order to compute $\beta_1 = \sum_{\l \in \SM{m,n}} \b$ using Hochster's Formula (\Cref{thm:hochster'sformula}), the crux of the proof lies in computing the number of zero and one-dimensional faces in $\D_\l$ for $\l \in \SM{m,n}$. To make our argument using this technique, we do the following: For each $\lambda \in \SM{m,n}$, using \Cref{lem:ranks}, we compute
\begin{enumerate}
     \item[I.] the number of zero-dimensional faces in $\D_\l$,
     \item[II.] $\rk M$, 
     \item[III.] $\beta_{1,\l} = \text{(I)}-\text{(II)}-1.$
\end{enumerate}

Of course, this only works as all but finitely many $\b$ are zero.
This section unfolds as follows. In \Cref{lem:complexisexact}, we eliminate all $\lambda \in \SM{m,n}$ with $\beta_{1,\l} =0$. In \Cref{lem:3specialcases}, we compute $\beta_{1,\l}$ for three specific choices of $\l$, as these require distinct techniques. Next is the main result, \Cref{thm:main}, which computes all remaining non-zero $\b$ thus giving us $\mu(\IM{m,n})$ for $m \geq 3$ and $n\neq 4$. Finally, in Corollaries \ref{cor:mn=24etc} and \ref{cor:mn=23}  we obtain $\mu(\IM{m,n})$ for $m=2, n \geq 4$, and $(m,n)=(2,3)$, respectively. For ease of reading, some repetitive computations are moved to the appendix. 
\end{remark}

\begin{lemma}\label{lem:complexisexact}
    The complex $\tilde C (\D_\l) $ is exact at $C_0$ for $\l \leq 2e+1$ and $\l \geq 4e-1.$
\end{lemma}
\begin{proof}
    When $\l \leq 2e-1$, then $\l$ is either zero or one of the minimal generators of $\SM{m,n}$. So $\D_\l = \{\varnothing, \{\l\} \}$. Next, $\D_{2e}=\{\varnothing, \{e\} \}$ and $\D_{2e+1}=\{\varnothing, \{e\}, \{e+1\},\{e,e+1\}\}$. In each case, $\b=0$.

For $\l \geq 4e-1,$ write $\l= 4e-1+x$ for some $x \geq 0$. First, consider the case $(m,n) \neq (2,3).$ For $j \in \JM{m,n} \setminus \{0\}$, since $ e+x-j-1 \geq 0$, we get 
\[ \l - e-(e+j) = e+ (e+x-j-1) \in \SM{m,n} \]
unless 
\begin{enumerate}
    \item $e+x-j-1 = m$, i.e., $j=e+x-m-1 \in \JM{m,n}$, or
    \item $e+x-j-1 =n$, i.e., $j= e+x-n-1 \in \JM{m,n}.$
\end{enumerate}

If neither (1) nor (2) is true, then $\{e,e+j\} \in \D_\l$ for all $j \in \JM{m,n} \setminus \{0\},$ giving us $e-3$ one-dimensional faces of the same form as in \Cref{lem:ranks} (3). If either (1) or (2) happens, then $\{e, e+(e+x-m-1)\}$ or $\{e, e+(e+x-n-1)\}$ is not in $\D_\l$, respectively. However, we may replace these faces with $\{e+1, e+(e+x-m-1)\}$ or $\{e+y, e+(e+x-n-1)\}$ for some $y\in [1,n-1]\setminus\{m,n-m\}$, respectively, still giving us $e-3$ faces as in \Cref{lem:ranks} (3). In either of the cases, we get $\rk M \geq e-3.$ 
In each case, apply \Cref{lem:ranks} (2) to get $e-3 \leq \rk M \leq  \ell-1 \leq e-3$, where $\ell$ denotes the number of zero-dimensional faces. Thus, $\b = 0$.

Now let $(m,n) = (2,3).$ Using the same arguments as above, with $\l= 4e-1+x$ for some $x \geq 0$, we get that $\{e,e+j\} \in \D_\l$ unless either (1) or (2) above is true or $\l-e-(e+j) =F(\SM{2,3}) =2e+3$, i.e., $j=x-4 \in \JM{2,3} \setminus \{0\}.$ 
Since the previous paragraph's reasoning still works in the event of $(1)$ and $(2)$, we only need to check what happens when $j = x-4 \in \JM{m,n} \setminus \{0\}.$ Note that this happens only when $x \geq 4$ and $x \neq 6,7.$ In that case, while $\{e,e+(x-4)\} \notin \D_\l$, we have $\{e+y, e+(x-4)\} \in \D_\l$ for some $y \in [4,e-1] \bs \{x-4\}$. This still gives us $e-3$ one-dimensional faces and repeating the same arguments as above implies that $\b=0.$
\end{proof}

We now want to compute the non-zero $\b$, so necessarily, $2e+2 \leq \l \leq 4e-2$.

\begin{remark}\label{rem:Step1}
    For $2e+2 \leq \l \leq 4e-2$, write $\l = 2e+2 +x$ where $0 \leq x \leq 2e-4.$ To simplify the computations for step(I) of \Cref{rem:ideaofproof}, set $$A=[0,x+2] \bs \{m,n\}, \q B=\{x+2-m,x+2-n\}.$$ 
Then for $(m,n) \neq (2,3)$ we have
\begin{align*}\label{count1dfaces4}
      \{e+j\} \in \D_\l 
      &\Longleftrightarrow (2e+2+x)-(e+j) = e + (x+2-j) \in \SM{m,n} \\  \nonumber
      &\Longleftrightarrow j \in  (\JM{m,n} \cap A) \bs B.
 \end{align*}
Let $B'= \{x,x-1,x-1-e\}$. Note that when $(m,n) = (2,3),$ then $\{e+j\} \in \D_\l \iff j \in (\JM{m,n} \cap A) \bs B'.$
\end{remark}

\begin{lemma}\label{lem:3specialcases}
Let $2 \leq m < n \leq e-2.$ Then
\begin{enumerate}[{\rm(1)}]
    \item 
    $$\beta_{1,2e+2m}= \begin{cases}
 m-1, & 2 \leq m \leq \ll \frac{n-1}{2} \rr, \\
     m-2, & \lc \frac{n}{2} \rc \leq m \leq \ll \frac{e-1}{2} \rr,\\
     e-2-m, & \left\lceil \frac{e}{2} \right\rceil \leq m \leq e-3.
 \end{cases}$$ 
\item   $$  \beta_{1,2e+m+n}= \begin{cases}
      \ll \frac{m+n}{2} \rr - 1, & 2 \leq m \leq e-n-1, \\
     e-3 -  \ll \frac{m+n-3}{2} \rr, & e-n \leq m \leq e-3.
 \end{cases}$$
\item  $$\beta_{1,2e+2n}= \begin{cases}
     n-2, & 3 \leq n \leq \ll \frac{e-1}{2} \rr, \\
     e-n-2, & \left\lceil \frac{e}{2} \right\rceil \leq n \leq \ll \frac{e+m-1}{2} \rr, \\
     e-n-1, & \lc \frac{e+m}{2} \rc \leq n \leq e-2.
 \end{cases}$$  
In particular, for $(m,n) = (2,3)$ we have
 \begin{align*} 
 \beta_{1,2e+4} = 0, \q \beta_{1,2e+5} = 1, \q \text{ and } \q \beta_{1,2e+6} = 
 \begin{cases}
     0, & e=5, \\
     1, & e \geq 6.
 \end{cases} 
 \end{align*} 
 \end{enumerate}
\end{lemma}

\begin{proof}
First assume that $(m,n) \neq (2,3).$ We use \Cref{rem:Step1} to compute the number of zero-dimensional faces.

\noindent
(1) For $\lambda=2e+2m$, write $\lambda=2e+2+x$ with $x=2m-2$. For step(I), we have
\begin{align*}
    2 \leq m \leq \ll \frac{n-1}{2} \rr, & \q | A \cap \JM{m,n} \setminus B| = \big\lvert [0,2m] \bs \{m\} \big\rvert = 2m, \\
    \lc \frac{n}{2} \rc \leq m \leq \ll \frac{e-1}{2} \rr, & \q |A \cap \JM{m,n} \bs B| =  \big\lvert [0,2m] \bs \{m,n,2m-n\} \big\rvert =2m-2, \\
    \left\lceil \frac{e}{2} \right\rceil \leq m \leq e-3, & \q |A \cap \JM{m,n} \bs B|  = \big\lvert \JM{m,n} \bs \{2m-n\} \big\rvert =e-3.
\end{align*}

For step(II), we refer to the details in Appendix \ref{x=2m-5} and have
$$\im \d_1 \cong \begin{cases}
    \Z^{m}, & 2 \leq m \leq \ll \frac{n-1}{2} \rr, \\
         \Z^{m-1}, & \lc \frac{n}{2} \rc \leq m \leq \left\lfloor \frac{e-1}{2} \right\rfloor,  \\
         \Z^{m-2}, & \left\lceil \frac{e}{2} \right\rceil \leq m \leq e-3.
\end{cases}$$

In step(III), we obtain  
\[
 \boxed{\beta_{1,2e+2m}= \begin{cases}
 m-1, & 2 \leq m \leq \ll \frac{n-1}{2} \rr, \\
     m-2, & \lc \frac{n}{2} \rc \leq m \leq \ll \frac{e-1}{2} \rr, \\
     e-2-m, & \left\lceil \frac{e}{2} \right\rceil \leq m \leq e-3.
 \end{cases}} 
 \] 
 
\noindent
(2) For $\lambda=2e+m+n$, that is $x=m+n-2,$  we get the following: 
\begin{align*}
   2 \leq m \leq e-n-1, & \q |A \cap \JM{m,n} \bs B| = \big\lvert [0,m+n] \bs \{m,n\} \big\rvert = m+n-1, \\
e-n \leq m \leq e-3, & \q |A \cap \JM{m,n} \bs B| = |\JM{m,n}| = e-2.
\end{align*}

With details in Appendix \ref{x=m+n-5}, we have  
$$\im \d_1 \cong \begin{cases}
    \Z^{\ll \frac{m+n-1}{2} \rr}, & 2 \leq m \leq e-n-1, \\
    \Z^{\ll \frac{m+n-3}{2} \rr}, & e-n \leq m \leq e-3.
\end{cases}$$

Thus in step(III), we obtain  
\[ \boxed{\beta_{1,2e+m+n}= \begin{cases}
      m+n-2-\ll \frac{m+n-1}{2} \rr = \ll \frac{m+n}{2} \rr - 1, & 2 \leq m \leq e-n-1, \\
     e-3 -  \ll \frac{m+n-3}{2} \rr, & e-n \leq m \leq e-3.
 \end{cases}} 
 \] 
 
\noindent
(3) For $\lambda=2e+2n,$ that is $x=2n-2$ we have
  \begin{align*}
 3 \leq n \leq \ll \frac{e-1}{2} \rr, & \q |A \cap \JM{m,n} \bs B| = \big\lvert [0,2n] \bs \{m,n,2n-m\} \big\rvert = 2n-2, \\
\lc \frac{e}{2} \rc \leq n \leq \ll \frac{e+m-1}{2} \rr, & \q |A \cap \JM{m,n} \bs B| = \big\lvert \JM{m,n} \bs \{2n-m\} \big\rvert = e-3, \\
\lc \frac{e+m}{2} \rc \leq n \leq e-2, & \q |A \cap \JM{m,n} \bs B| = |\JM{m,n}|= e-2.
  \end{align*}

With details in Appendix \ref{x=2n-5}, we have
$$\im \d_1 \cong \begin{cases}
    \Z^{n-1}, & 3 \leq n \leq \ll \frac{e-1}{2} \rr, \\
    \Z^{n-2}, & \lc \frac{e}{2} \rc \leq n \leq  e-2.
\end{cases}$$

Thus in step(III), we obtain
\[ \boxed{\beta_{1,2e+2n}= \begin{cases}
     n-2, & 3 \leq n \leq \left\lfloor \frac{e-1}{2} \right\rfloor, \\
     e-n-2, & \left\lceil \frac{e}{2} \right\rceil \leq n \leq \ll \frac{e+m-1}{2} \rr, \\
     e-n-1, & \lc \frac{e+m}{2} \rc \leq n \leq e-2.
 \end{cases}} 
 \] 

\noindent
 Note that when $(m,n) = (2,3)$, then $B' = B \cup \{x-1-e\} = B$ as $x$ is either $2,3,$ or $4$ in the above cases. Thus, the same computations as above hold and for $e \geq 6,$ we have $\beta_{1,2e+4} = m-2 = 0,$ $\beta_{1,2e+5} = \ll \frac{5}{2} \rr - 1 = 2-1 = 1,$ and $\beta_{1,2e+6} = n-2 = 1.$
\end{proof}

\begin{theorem}\label{thm:main}
    Let $\SM{m,n}$ be a Sally type numerical semigroup where $3 \leq m < n \leq e-2$ but $(m,n) \neq (3,4)$. Then the number of minimal generators of $\IM{m,n}$ is $$\mu(\IM{m,n}) = 
    {e-2 \choose 2} - 2.$$
\end{theorem}
\begin{proof}
     \Cref{lem:complexisexact} implies that $\b =0$ for $\l \leq 2e+1$ and $\l \geq 4e-1$, and \Cref{lem:3specialcases} gives us $\b$ for $\l \in \{2e+2m,2e+m+n,2e+2n\}$. So suppose $ 2e+2 \leq \l \leq 4e-2$, $\l \not\in \{2e+2m, 2e+m+n,2e+2n\}.$ Write $\l = 2e+2 +x$, where $0 \leq x \leq 2e-4$, and $x \not\in \{2m-2,m+n-2,2n-2\}.$ We compute the number of zero-dimensional faces in $\D_\l$ which is the  step(I) of \Cref{rem:ideaofproof}, using \Cref{rem:Step1} and consider each of the following cases separately. 
 \begin{enumerate}[(a)]
     \item If $0 \leq x \leq m-3$, then $x+2 \leq m-1$ and so
     $$|A \cap \JM{m,n} \bs B| = |A| = \big\lvert [0,x+2] \big\rvert =x+3.$$
     \item If $m-2 \leq x \leq n-3,$ then as $x \neq 2m-2,$ we get
     $$|A \cap \JM{m,n} \bs B| = \big\lvert A \bs \{x+2-m\} \big\rvert = \big\lvert [0,x+2]  \bs \{m,x+2-m\} \big\rvert =x+1.$$
     \item If $n-2 \leq x \leq e-3,$ then as $x \not\in \{2m-2,m+n-2,2n-2\},$ we get
     $$|A \cap \JM{m,n} \bs B| = |A \bs B| = \big\lvert [0,x+2] \bs \{m,n, x+2-m,x+2-n\} \big\rvert =x-1.$$
     \item If $e-2 \leq x \leq e+(m-1)-2,$ then $$|A \cap \JM{m,n} \bs B| = |\JM{m,n} \bs B| = e-4.$$
     \item If $e+m-2 \leq x \leq e+(n-1)-2,$ then $$|A \cap \JM{m,n} \bs B| = |\JM{m,n} \bs \{x+2-n\}| = e-3.$$
     \item If $e+n-2 \leq x \leq e+(e-4),$ then $$|A \cap \JM{m,n} \bs B| = |\JM{m,n}| = e-2.$$
 \end{enumerate}
 
Next, we compute the rank of the matrix $M$ as step(II) of \Cref{rem:ideaofproof} . Note that any one-dimensional face is a superset of some zero-dimensional face. Set $\I = (\JM{m,n} \cap A) \bs B.$ When $0 \leq x \leq e-3$, for $j<l \in \I$, $$\l -(e+j)-(e+l) = x+2-j-l < e-1.$$
 So the only one-dimensional faces in this case are $\{e+j,e+l\}$ for some $j<l \in \I$ with $j+l=x+2$. In essence, we pair up the zero-dimensional faces satisfying this condition to get one-dimensional faces. Thus, we have the following: 
 \begin{enumerate}[(a)]
     \item for $0 \leq x \leq m-3$, $\rk M = \ll \frac{x+3}{2} \rr,$ 
     \item for $m-2 \leq x \leq n-3,$ $\rk M = \ll \frac{x+1}{2} \rr,$ 
     \item for $n-2 \leq x \leq e-3,$ $\rk M = \ll \frac{x-1}{2} \rr.$ 
 \end{enumerate}
 
When $e-2 \leq x \leq 2e-4,$ for one-dimensional faces of $\D_\l$ of the form $\{e,e+j\},$ $j \in \I,$ we must have 
\begin{align*}
\{e,e+j\} \in \D_\l &\iff  \l -e-(e+j) = x+2-j \in \SM{m,n}\\
&\iff x+2-j \geq e, \ x+2-j \neq e+m, \text{ and } x+2-j \neq e+n \\
& \iff j \in \left([1,x+2-e] \cap \I \right) \setminus \{ x+2-e-m, x+2-e-n \}.
\end{align*}

For (d), when $x=e-2$, no such face is possible. When $e-2 < x \leq e + (m-1) -2$, none of $m,n,x+2-e-m,x+2-e-n,$ and elements of $B$ are in $[1,x+2-e],$ giving us $x+2-e$ such faces. Thus, $\rk M \geq x+2-e$.

For (e), if $x=e+m-2,$ then none of $n,x+2-n, x+2-e-n,x+2-e-m=0 $ are in $[1,x+2-e]$, but $m$ is. If $x>e+m-2,$ then none of $n,x+2-n, x+2-e-n$ are in $[1,x+2-e]$ but $m$ and $x+2-e-m$ are. If $x+2-e-m=m$, we get $x+1-e$ faces. Otherwise, choosing 
\begin{equation}\label{eq:(e)chooseyform=2}
    y \in [1,m-1] \setminus \{x+2-e-m\}
\end{equation} the face $\{e+y,e+(x+2-e-m)\}$ increases $\rk M$ by one. Thus, $\rk M \geq x+1-e$.

For (f), we begin by assuming $x \notin \{e+2m-2,e+2n-2,e+m+n-2\}.$
If $x=e+n-2$, then $m,n,x+2-e-m$ are all distinct elements of $[1,x+2-e]$, so $\rk M\geq x-1-e.$ Pick 
\begin{equation}\label{eq:(f)chooseyfor(m,n)=(2,3)}
    y\in [1,m-1] \bs \{x+2-e-m\} \, (\neq \emptyset \text{ as } (m,n) \neq (2,3))
\end{equation}
to get an additional one dimensional face $\{e+y,e+(x+2-e-m)\},$ which implies that $\rk M \geq x-e$.

If $x>e+n-2$, then $m,n,x+2-e-m,x+2-e-n$ are all distinct in $[1,x+2-e]$ so $\rk M \geq x-2-e$. Since $\{e,e+(x+2-e-m)\}$ and $\{e,e+(x+2-e-n)\}$ are not one dimensional faces of $\D_\l$, pick 
\begin{equation}\label{eq:(f)pickyform=2ormn=34}
    y \in [1,m-1] \bs \{x+2-e-m,x+2-e-n\}
\end{equation}

to get the additional faces $\{e+y,e+(x+2-e-m)\}, \{e+y,e+(x+2-e-n)\}$, implying $\rk M \geq x-e$. 

 We now look at the special cases, $x \in \{e+2m-2,e+2n-2,e+m+n-2\}.$
If $x=e+2m-2$, then $x+2-e-m=m$ and hence $\rk M\geq x-1-e.$ In this case, pick 
\begin{equation}\label{eq:(f)pickyformn=34}
    y \in [1,m-1] \bs \{n-m, x+2-e-n\}
\end{equation}

to get $\{e+y,e+(x+2-e-n)\} \in \D_\l$ which implies that $\rk M \geq x-e.$ 
Similarly, if $x=e+2n-2$, then $x+2-e-n=n$, so $\rk M\geq x-1-e$ and we pick $y \in [1,m-1] \bs \{x+2-e-m\}$ to get $\rk M \geq x-e.$
If $x=e+m+n-2$, then $m=x+2-e-n$ and $n=x+2-e-m$ so $\rk M\geq x-e.$

To conclude, in part (f) we have $\rk M \geq x-e$.

We now look for any other one-dimensional faces that add to the rank of $M.$ In particular, the faces of the form $\{e+j,e+l\}$ where $j,l \in \JM{m,n} \setminus \{0\}$ with $j \geq x+2-e+1$. This gives us the following $\ll \frac{2e-x-3}{2} \rr$  pairs of $(j,l)$:

\begin{figure}[!h]
    \centering
\begin{tikzpicture}
    \node at (4,3) {$x+2-e+1 < x+2-e+2 < \cdots < e-2 < e-1$};
    \draw (3.5,3.25) -- (3.5,3.50);
   \draw (3.5,3.50) -- (6.5,3.50);
    \draw (6.5,3.50) -- (6.5,3.25);
    \draw (1,2.75) -- (1,2.5);
    \draw (1,2.5) -- (7.75,2.5);
    \draw (7.75,2.5) -- (7.75, 2.75);
\end{tikzpicture}
\end{figure}

In (d), $m,n,x+2-m,x+2-n$ are in the above sequence so we get $\ll \frac{2e-x-3}{2} \rr-2$ faces. In (e), $n$ and $x+5-n$ are in the sequence but $m$ and $x+5-m$ are not, giving $\ll \frac{2e-x-3}{2} \rr-1$ faces. Finally, in (f), none of $m,n,x+5-n,x+5-m$ are in the sequence, so we get $\ll \frac{2e-x-3}{2} \rr$ faces.

Thus, in all of the cases (d), (e), and (f) above, that is, for $e-2 \leq x \leq 2e-4,$ we get $\rk M = \ll \frac{x-3}{2} \rr $.\\

Now we substitute $x=\l-2e-2$ and subtract 1 from the difference of 0 dimensional faces and rank of $M$ as step (III)  to compute the respective $\b$. 
\begin{align}  \label{mainBox}
\boxed{\beta_{1,\l}= 
\begin{cases}
    \ll \frac{\lambda}{2} \rr-e, & 2e+2 \leq \l \leq 2e+m-1, \\
    \ll \frac{\lambda}{2} \rr-e-1, & 2e+m \leq \l \leq 2e+n-1, \\
    \ll \frac{\lambda}{2} \rr-e-2, & 2e+n \leq \l \leq 3e-1, \\
    2e-3 - \ll \frac{\lambda-1}{2} \rr, & 3e \leq \l \leq 3e+m-1, \\
   2e -2 - \ll \frac{\lambda-1}{2} \rr, & 3e+m \leq \l \leq  3e +n-1, \\
   2e-1 - \ll \frac{\lambda-1}{2} \rr , & 3e+n \leq \l \leq 4e-2.
 \end{cases}} 
 \end{align} 

To complete the proof, we add the required $\b$ from the box above and \Cref{lem:3specialcases}. Note that from box \ref{mainBox}, we get  
\[ \sum_{\l=2e+2}^{4e-2} \b  = \frac{e^2}{2} - \frac{5e}{2}  = {e-2 \choose 2} -3,  \]
see Appendix \ref{sum-betti-part2} for details.

Now note that the three special cases of $\l$ computed in \Cref{lem:3specialcases} still appear in the intervals in $\Cref{mainBox}$. To get the final $\beta_1$ we only need to subtract the value of $\b$ from box \ref{mainBox} for $\l \in \{2e+2m,2e+m+n,2e+2n\}$ and add the corresponding $\b$ from \Cref{lem:3specialcases}.  We will now show that the cumulative difference corresponding to these values of $\l$'s is $1$, which gives us the conclusion of our theorem.

Observe that $\beta_{1, 2e+2m}$ and $\beta_{1, 2e+2n}$ computed above and in \Cref{lem:3specialcases} are the same. The reason for computing them separately is that the numbers of zero-dimensional faces and the two kinds of one-dimensional faces contributing to $\rk M$ differ from those in Part 2, despite the final difference being the same. 
So, the only contribution is made by $\beta_{1,2e+m+n}$. Note that $2e+m+n$ can only satisfy the inequalities in (c) and (d) above. If $2e+m+n$ satisfies (c), i.e., $2e+n \leq 2e+m+n \leq 3e-1$, then $m \leq e-n-1$ giving us
$$m+n-2 - \ll \frac{m+n-1}{2} \rr - \left(\ll \frac{2e+m+n-1}{2} \rr -e-2 \right) = 1.$$  
If $3e \leq 2e+m+n \leq 3e+m-1$, then $m \geq e-n$ implying that the difference 
$$e-3- \ll \frac{m+n-3}{2} \rr - \left(2e-3- \ll \frac{2e+m+n-1}{2} \rr \right) =1.$$
\end{proof}

\begin{corollary}\label{cor:mn=24etc}
    Suppose $m=2$ or $(m,n)=(3,4)$ but $(m,n) \neq (2,3)$. Then 
    $$\mu(\IM{m,n}) =\begin{cases}
    {e-2 \choose 2}, &  {\rm if} \ (m,n) \in \{(2,4),(2,5),(3,4)\}, \\
    {e-2 \choose 2} -1, & {\rm if} \ m=2 \text{ and } n \geq 6.
    \end{cases} $$
\end{corollary}
\begin{proof}
    The proof differs from that of \Cref{thm:main} only in the following instances:

    In part (e), when $m=2$ and $x=e+1,$ the additional pair obtained from the choice of $y$ in $\eqref{eq:(e)chooseyform=2}$ does not exist. Thus, $\rk M \geq x-e$ when $m=2$ and $x=e+1.$

    In part (f), if $(m=2 \text{ or } (m,n)=(3,4)) \text{ and } (x=e+n-1)$, the set in the equation \eqref{eq:(f)pickyform=2ormn=34} is empty. But, as $x+2-e-n=1$, $\{e+1=e+(x+2-e-n), e+(x+2-e-m)\} \in \D_\l$ giving us $\rk M \geq x-1-e$. When $(m,n) \in \{(2,4),(2,5),(3,4)\}$, no other one-dimensional face contains $e+(x+2-e-m)$ or $e+(x+2-e-n).$ But, when $m=2$ and $n \geq 6$, pick $y' \in [1,n-3] \bs \{1,2\}$ to get a one-dimensional face $\{e+y',e+(x+2-e-n)\}$ that adds an additional one to $\rk M$. 

    Also, in part (f), if $(m,n)=(3,4)$ and $x = e+2m-2 = e+4,$ then the additional face, as chosen using $\eqref{eq:(f)pickyformn=34}$ doesn't exist and so in this case, $\rk M \geq x-1-e.$

    In conclusion, to the $\mu(\IM{m,n})$ obtained in \Cref{thm:main}, an additional 2 is added when $(m,n) \in \{ (2,4), (2,5), (3,4)\}$ and an additional $1$ is added when $m=2$ and $n \geq 6$. 
\end{proof}

\begin{corollary}\label{cor:mn=23}
    When $(m,n)=(2,3)$, $\mu(\IM{m,n})= {e - 2 \choose 2} -1$.
\end{corollary}

\begin{proof}
    From \Cref{lem:complexisexact}, we know that $\b=0$ for $\l \leq 2e+1$ and for $\l \geq 4e-1.$ Note that in \Cref{rem:Step1}, $B'=B \cup \{x-1-e\}=B$ for $x \leq e.$ Thus, for $2e+2 \leq \l \leq 3e+2$, the calculations of $\b$ are as in \Cref{thm:main}. This takes care of parts $(a)-(e)$ above. One needs to remember here that $2e+3$ is the Frobenius number in this case (see \Cref{prop:Frob of semigroup}). So, $\beta_{1,2e+3}$ is not included in the sum.

    Now, in part (f), that is, for $e+1 \leq x \leq 2e-4,$ using \Cref{rem:Step1} we obtain $\{e+j\} \in \D_\l$ for $j \in  \JM{m,n} \bs \{x-1-e\}$, giving us $e-3$ zero-dimensional faces. However, when $x \in \{e+3,e+4\}$, $x-1-e \not\in \JM{m,n},$ giving us $e-2$ zero-dimensional faces. 
    
    When $x \in \{e+3,e+4\}$, the computation for $\rk M$ works as in the proof of part (f) in \Cref{thm:main}, thus giving us the same value of $\b.$ When $x \not\in \{e+3,e+4\}$ and $e+1 \leq x \leq 2e-4$, the one-dimensional faces of the form $\{e+j,e+l\},$ where $j+l=x+2$, remain the same. But as $\{e,e+(x-1-e)\} \notin \D_\l$ in this case, there are only $x-e-1$ one-dimensional faces of the form $\{e,e+j\}$ now. However, while computing the $\b$, we get the same number as in part (f) of \Cref{thm:main} since the difference (as in step(III)) remains unchanged. Basically, the number of zero-dimensional faces decreases by one, but so does $\rk M$. 

    All this amounts to subtracting $1= \beta_{1,2e+3}$ from $\mu(\IM{m,n})$ obtained in \Cref{thm:main} and adding 2 as we saw in the proof of \Cref{cor:mn=24etc}, thus giving us the required result.
\end{proof}

Combining the above results, we present our main theorem: a concise formula for the minimal number of generators of the defining ideal for a Sally type numerical semigroup $\SM{m,n}$.

\maintheorem
\begin{proof}
    This follows directly from \Cref{thm:main}, \Cref{cor:mn=24etc}, and \Cref{cor:mn=23}.
\end{proof}

\begin{corollary}\label{widthconj}
    Sally type numerical semigroups of the form $\SM{m,n}$ satisfy Herzog and Stamate's width conjecture, i.e., $$\mu(\IM{m,n}) \leq {\wid{(\SM{m,n})} +1 \choose 2}.$$ 
    
\end{corollary}
\begin{proof} This can be seen using \Cref{thm:maintheorem} and the fact that $\wid{(\SM{m,n})} = 2e-1-e=e-1$.
Thus, 
\[\mu(\IM{m,n}) \leq {e-2 \choose 2 }\leq {\wid{(\SM{m,n})} +1 \choose 2}. \hskip\fill\qedhere \]
\end{proof}

\section{A minimal Generating Set for $\IM{m,n}$}
\label{section:minimalgenerators}
Recall that for a field $k,$ the semigroup ring $k[\SM{m,n}] \cong k[X_i \mid i \in \JM{m,n}]/\IM{m,n},$ where $\IM{m,n}$ is the defining ideal. In the previous section, we computed $\mu(\IM{m,n}),$ the number of minimal generators of $\IM{m,n}.$ In this section, we give a list of a minimal generating set for this defining ideal for $e \geq 10.$

In the polynomial ring $R:=k[X_0,X_1,\ldots, X_{e-1}],$ with grading on the variables given by $\deg X_i = e+i$ for all $i \in [0,e-1]$, consider the following binomials:

\vspace{-1cm}
\begin{multicols}{2}
    \begin{align*}
    \fsy_1(j,k) &= X_{k+1}X_{j}-X_{k}X_{j+1} \\
    \fsy_2(j) &= X_jX_0^2-X_{j+1}X_{e-1} \\
    \fsy_3(j)  &= X_{m+2}X_{j-4}-X_{m-2}X_{j}  \\ 
    \fsy_4(j)  &= X_{m-1}X_{j+2}-X_{m+1}X_j   \\
    \fsy_5(j)  &= X_{n-1}X_{j+2}-X_{n+1}X_j   \\
    \fsy_6(j,k)  &= X_0X_jX_{m-1}-X_{m+2+k}X_{e-3+j-k} \\
    \fsy_7(j,k)  &= X_0X_{1+k}X_{n-1-k}-X_{n+2-j}X_{e-2+j} \\
    \fsy_8(j)  &= X_0^2X_{n-1}-X_{n+2-j}X_{e-3+j}
    \end{align*}
    \vfill
    \columnbreak
    
    \begin{align*}
    \fsy_9  &= X_{m-1}^2-X_{m-4}X_{m+2} \\
    \fsy_{10}  &= X_{m+2}X_{e-4}-X_{m-1}X_{e-1} \\
    \fsy_{11}  &= X_{m+1}X_{n-2}-X_{m-2}X_{n+1}  \\
    \fsy_{12}  &= X_{m+2}X_{n-1}-X_{m-1}X_{n+2}   \\
%   \fsy_{13}(a) &= X_1^3-X_0^2X_3 \tiny{\srishti{or $f_{13}(a)=X_1^3-X_6X_{e-2}$}} \\
    \fsy_{13,a} &= X_1^3-X_6X_{e-2} \\
    \fsy_{13,b} &= X_1^3-X_4X_{e-1}   \\
%   \fsy_{14}&= X_1^2X_{n-1}-X_0^2X_{n+1} \tiny{\srishti{or 
    \fsy_{14} &= X_1^2X_{n-1}-X_{n+2}X_{e-1} \\
    \fsy_{15}(j)&=X_2^2X_{j+1}-X_{6+j}X_{e-1}
\end{align*}
\end{multicols}

Based on the values of $m$ and $n$, not all of these binomials exist in the polynomial ring $k[X_i \mid i \in \JM{m,n}]$. However, for each pair $(m,n),$ $2 \leq m < n \leq e-2$, there exists a subset of these binomials that generates $\IM{m,n}$ for that pair. Denote this set of minimal generators by $\cIM{m,n}$. In the following subsections, we explicitly write out the set $\cIM{m,n}$ for all $2 \leq m < n \leq e-2$, and prove in \Cref{thm:mingens} that it indeed minimally generates $\IM{m,n}$.

\subsection*{5.1}  $4 \leq m < n=m+1 \leq e-2$. 
$\cIM{m,m+1}$ consists of 

\begin{itemize}
    \item[(a)] ${e-4 \choose 2}$ elements $\{\fsy_1(j,k) \mid k<j \in [0,e-2] \bs [m-1,m+1] \},$ 
    \item[(b)] $e-4$ elements $\{\fsy_2(j) \mid j \in [0,e-2] \bs [m-1,m+1] \},$
    \item[(c)] 
    \begin{tabular}[t]{ll}
        $e-10 \text{ elements } \{\fsy_3(j) \mid j \in [4,m-1] \cup [m+6, e-1] \}$, & if $m \leq e-6$  \\
        $ e-9 \text{ elements } \{\fsy_3(j) \mid j \in [4, e-6] \}$, & if $m=e-5 $ \\
        $e-8 \text{ elements } \{\fsy_3(j) \mid j \in [4, e-5]\}$, & if $m=e-4$ \\
        $e-7 \text{ elements } \{\fsy_3(j) \mid j \in [4, e-4] \}$,  & if $m=e-3,$
    \end{tabular}
    \item[(d)] 
    \begin{tabular}[t]{ll}
        $5 \text{ elements } \{\fsy_6(0,0), \fsy_6(1,0), \fsy_7(0,1), \fsy_9,\fsy_{10} \}$, &  if $m \leq e-6$  \\
        $4 \text{ elements } \{\fsy_6(0,0), \fsy_6(1,0), \fsy_7(0,1), \fsy_9\}$, & if $m = e-5$ \\
        $3 \text{ elements } \{\fsy_6(1,0), \fsy_7(0,1), \fsy_9 \}$, & if $m =e-4$ \\
        $2 \text{ elements } \{\fsy_9,\fsy_{10}\}$,  & if $m=e-3.$
   \end{tabular}
\end{itemize}

\subsection*{5.2} $3 \leq m < n=m+2 \leq e-2$. 
$\cIM{m,m+2}$ consists of 

\begin{itemize}
    \item[(a)] ${e-5 \choose 2}$ elements $\{\fsy_1(j,k) \mid k <  j \in [0,e-2] \bs [m-1,m+2] \},$
    \item[(b)] $e-5$ elements $\{\fsy_2(j) \mid j \in [0,e-2] \bs [m-1,m+2]\},$
    
    \smallskip
    \item[(c)] 
    \begin{tabular}[t]{ll}
    $ \begin{cases}
    e-6 \text{ elements } \{\fsy_4(j) \mid j \in [0,e-3] \bs \{m-2,m-1,m,m+2\}\} & \\   
    e-7 \text{ elements } \{\fsy_5(j) \mid j \in [0,e-3] \bs [m-2,m+2] \}, 
    \end{cases} $
    &  if $m \leq e-5$ \vspace{1mm}\\
    \smallskip

    $ \begin{cases}
    e-5 \text{ elements } \{\fsy_4(j) \mid j \in [0,e-3] \bs [m-2,m]\} &  \\
    e-6 \text{ elements } \{\fsy_5(j) \mid j \in [0,e-3] \bs [m-2,m+1]\}, 
    \end{cases} $
    &  if $m=e-4,$ \\ 
    \end{tabular}
    \item[(d)] 
    \begin{tabular}[t]{ll}
        $4 \text{ elements } \{\fsy_6(0,-1), \fsy_6(1,-1), \fsy_7(0,0), \fsy_8(0)\}$, & if $m \leq e-6$ \\
        $4 \text{ elements } \{\fsy_6(0,-1), \fsy_6(1,-1), \fsy_7(0,0), \fsy_8(1)\}$, & if $m =e-5$ \\
        $2 \text{ elements } \{\fsy_6(1,-1), \fsy_7(1,0)\}$, & if $m =e-4$.
    \end{tabular}
\end{itemize}

\subsection*{5.3}  $3 \leq m < n-2 \leq e-4$. 
$\cIM{m,n}$ consists of 

\begin{itemize}
    \item[(a)] ${e-5 \choose 2}$ elements $\{\fsy_1(j,k) \mid k<j \in [0,e-2] \bs \{m-1,m,n-1,n\}\},$
    \item[(b)] $e-5$ elements $\{\fsy_2(j) \mid j \in [0,e-2] \bs \{m-1,m,n-1,n\}\},$

    \smallskip
    \item[(c)] 
    \begin{tabular}[t]{ll}
    \hskip-3mm
    $\begin{cases}
    e-7 \text{ elements } \{ \fsy_4(j) \mid j \in [0,e-3] \bs \{ m-2,m-1,m,n-2,n \} \} & \\
    e-8 \text{ elements } \{\fsy_5(j) \mid j \in [0,m-3] \cup [m+1,n-3] \cup [n+1, e-3]\}, 
    \end{cases} $
    & \hskip-2mm if $n \leq e-3$ \vspace{1mm}\\ 
    \smallskip
    
    \hskip-3mm
    $\begin{cases}
    e-6 \text{ elements } \{\fsy_4(j) \mid j \in [0,e-3] \bs \{m-2,m-1,m,n-2\}\} & \\
    e-7 \text{ elements } \{\fsy_5(j) \mid j \in [0,e-3] \bs \{m-2,m-1,m,n-2,n-1\}\}, 
    \end{cases} $
    & \hskip-2mm  if $n = e-2$, \\ 
    \end{tabular}
    \item[(d)]
    \begin{tabular}[t]{ll}
        $6 \text{ elements } \{\fsy_6(0,0), \fsy_6(1,0), \fsy_7(0,0), \fsy_8(0), \fsy_{11},\fsy_{12}\}$, & if $n \leq e-4$ \\
        $6 \text{ elements } \{\fsy_6(0,-1), \fsy_6(1,-1), \fsy_7(0,0), \fsy_8(1), \fsy_{11},\fsy_{12}\}$, & if $n =e-3$ \\
        $4 \text{ elements } \{\fsy_6(0,0),\fsy_6(1,-1), \fsy_7(1,0), \fsy_{11}\}$, & if $m =e-2$.
\end{tabular}
\end{itemize}

\subsection*{5.4} $m=2, n\in\{4,5\}$.
$\cIM{2,n}$ consists of 

\begin{itemize}
    \item[(a)] ${e-5 \choose 2}$ elements $\{\fsy_1(j,k) \mid k < j \in [0,e-2] \bs \{1,2,n-1,n\}\},$
    \item[(b)] $e-5$ elements $\{\fsy_2(j) \mid j \in [0,e-2] \bs \{1,2,n-1,n\}\},$
    \item[(c)] 
    \begin{tabular}[t]{ll}
        $\begin{cases}
        e-6 \text{ elements } \{\fsy_4(j) \mid j \in [3,e-3] \bs \{4\} \} \\
        e-7 \text{ elements } \{\fsy_5(j) \mid j \in [5,e-3]\}, 
        \end{cases}$
        & if $n=4$ \vspace{1mm}\\
        \smallskip
        
        $\begin{cases}
        e-7 \text{ elements } \{\fsy_4(j) \mid j \in [4,e-3] \bs \{5\}\} \\ 
        e-8 \text{ elements } \{\fsy_5(j) \mid j \in [6,e-3] \},
        \end{cases}$
        & if $n = 5,$ \\ 
    \end{tabular}
    \item[(d)] 
    \begin{tabular}[t]{ll}
        $6 \text{ elements } \{\fsy_6(0,-1), \fsy_6(1,-1), \fsy_7(0,0), \fsy_8(0), \fsy_{13,a},\fsy_{14}\}$, & if $n=4$ \\
        $8 \text{ elements } \{\fsy_6(0,0),\fsy_6(1,0), \fsy_7(0,0),\fsy_8(0),\fsy_{11},\fsy_{12},\fsy_{13,b},\fsy_{14}\}$, & if $n=5$.
    \end{tabular}
\end{itemize}

\subsection*{5.5} $(2,n), n \geq 6.$ 
$\cIM{2,n}$ consists of 

\begin{itemize}
    \item[(a)] ${e-5 \choose 2}$ elements $\{\fsy_1(j,k) \mid k < j \in [0,e-2] \bs \{1,2,n-1,n\}  \},$
    \item[(b)] $e-5$ elements $\{\fsy_2(j) \mid j \in [0,e-2] \bs \{1,2,n-1,n\}\},$
    \item[(c)] 
    \begin{tabular}[t]{ll}
        $\begin{cases}
        e-7 \text{ elements } \{\fsy_4(j) \mid j \in [0,e-3] \bs \{0,1,2,n-2,n\} \\   
        e-8 \text{ elements } \{\fsy_5(j) \mid j \in [0,e-3] \bs \{0,1,2,n-2,n-1,n\}, \end{cases}$ 
        & if $n \leq e-3$ \vspace{1mm}\\
        \smallskip
        
        $\begin{cases}        e-6 \text{ elements } \{\fsy_4(j) \mid j \in [0,e-3] \bs \{0,1,2,n-2\} \\ 
        e-7 \text{ elements } \{\fsy_5(j) \mid j \in [0,e-3] \bs \{0,1,2,n-2,n-1\},
        \end{cases}$ 
        &  if $n = e-2$,
    \end{tabular}
    \item[(d)] 
    \begin{tabular}[t]{ll}
     $7 \text{ elements } \{\fsy_6(0,0), \fsy_6(1,0), \fsy_7(0,0), \fsy_8(0), \fsy_{11}, \fsy_{12}, \fsy_{13,b}\}$, & if $n \leq e-4$ \\
     $7 \text{ elements } \{\fsy_6(0,1), \fsy_6(1,0), \fsy_7(1,0), \fsy_8(1), \fsy_{11}, \fsy_{12}, \fsy_{13,b}\}$,   & if $n=e-3$  \\
        $5 \text{ elements } \{\fsy_6(0,0), \fsy_6(1,1), \fsy_7(1,0), \fsy_{11},\fsy_{13,b}\}$, &  if $n=e-2.$ 
    \end{tabular}
\end{itemize}

\subsection*{5.6} $(m,n) \in \{ (2,3), (3,4) \}.$  
$\cIM{m,n}$ consists of 

\begin{itemize}
    \item[(a)] ${e-4 \choose 2}$ elements $\{\fsy_1(j,k) \mid k < j \in  [0,e-2] \bs \{m-1,m,n\} \},$
    \item[(b)] $e-4$ elements $\{\fsy_2(j) \mid j \in [0,e-2] \bs \{m-1,m,n\}\},$
    \item[(c)] 
    \begin{tabular}[t]{ll}
        $e-8 \text{ elements } \{\fsy_3(j) \mid  j \in [8, e-1]\}$, & if $(m,n) = (2,3)$ \\ 
        $e-9 \text{ elements } \{\fsy_3(j) \mid j \in [9, e-1]\}$, & if $(m,n) = (3,4),$
    \end{tabular}
    \item[(d)] 
    \begin{tabular}[t]{ll}
        $4 \text{ elements } \{\fsy_6(0,1), \fsy_6(1,1), \fsy_{10},\fsy_{13,b}\}$, & if $(m,n) = (2,3)$ \\
        $6 \text{ elements } \{\fsy_6(0,1),\fsy_6(1,1), \fsy_7(0,1), \fsy_{10},\fsy_{15}(0),\fsy_{15}(1)\}$, & if $(m,n) = (3,4)$. 
    \end{tabular}
\end{itemize}

\begin{remark}\label{rem:monomialinuniquebinomial}
    Note that in every case, each binomial has a unique monomial that does not appear in any other binomial in that list.
\end{remark}

\begin{theorem}\label{thm:mingens}
    $\IM{m,n}$ is generated by the binomials in $\cIM{m,n}.$
\end{theorem}
\begin{proof}
    Using \Cref{thm:maintheorem}, it can be seen that $\cIM{m,n}$ consists of the right number of binomials to be a generating set. We first show that these binomials are, in fact, minimal generators of $\IM{m,n}$.
    
    Note that with the grading given by $\deg X_i = e+i$ for all $i \in [0,e-1],$ each $\eta \in \cIM{m,n}$ is a homogeneous binomial and thus, is in $\IM{m,n}$. Let $\mathfrak m = (\{X_i \mid i \in \JM{m,n}\})$ denote the maximal ideal of $k[X_i \mid i \in \JM{m,n}]$.  Suppose there is a binomial $\eta \in \cIM{m,n}$ such that $\eta \in \mathfrak m \IM{m,n}$. Then one can write    \begin{equation}\label{eq:binomialinmI}
        \eta = \eta_1 - \eta_2 = \sum_{j=1}^\ell r_j (\tau_1^j - \tau_2^j), 
    \end{equation}
    for some positive integer $\ell$, where, $(\tau_1^j - \tau_2^j)$ are binomials in $\IM{m,n}$ and $r_j \in \mathfrak m,$ for all $1 \leq j \leq \ell.$ Since $\tau_1^j - \tau_2^j\in \IM{m,n}$, each $\tau_i^j$ must have degree $\geq 2$ in the standard grading. Next, observe that at least one $\eta_i$ is quadratic, say $\eta_1$. Then $$\eta_2 + \sum_{j=1}^\ell r_j \tau_1^j - \sum_{j=1}^\ell r_j \tau_2^j$$
    is also quadratic. This is a contradiction as the degree of each monomial $r_j \tau_i^j$ is $\geq 3$. Thus, $\cIM{m,n} \subseteq \IM{m,n} \bs \mathfrak{m}\IM{m,n}.$ It remains to show that in each case above, the binomials in the set $\cIM{m,n}$ cannot be obtained as a linear combination of other ones in the same set. From \Cref{rem:monomialinuniquebinomial}, we know that in each case, every binomial has a unique monomial, either of degree 2 or 3. Moreover, if the said unique monomial is of degree 3, then no degree 2 monomial factors of it appear in any other binomial, except in the case of $f_2(j)$ where one sees an overlap of such factors with $f_1(j,k).$ Observe that $f_1$ and $f_2$ together constitute minors of a generic matrix; thus forming a subset of a minimal generating set. The presence of unique monomials, hence, concludes the argument.
\end{proof}

\section{A GAP-Algorithm} \label{section:algorithm}
In this section, we give a GAP code for a function named \texttt{GradedBetti} that, for given (imput) values of the multiplicity $e$, and the drops $m$ and $n,$ outputs all the graded Betti numbers $\b$ (using \Cref{lem:complexisexact}) and the first Betti number $\beta_1 = \sum_{\l} \b$ of $\SM{m,n}.$

The code first constructs the semigroup \texttt{S} for the given values of \texttt{e,m} and \texttt{n}. Then, for $2e+2 \leq \lambda \le 4e-2,$ the graded Betti numbers are computed. For each $\lambda$, the lists \texttt{ZD} and \texttt{OD} store the zero- and one-dimensional faces, respectively. Using \Cref{rem:ideaofproof}, the graded Betti number is computed and stored as \texttt{GB}.

{\small \begin{verbatim}
GradedBetti := function(e,m,n)
local L, S, Betti, lam, i, j, ZD, OD, Mat, Rk, GB;
L := [e..(2*e-1)];;
Remove(L,m+1);;
Remove(L,n);;
S := NumericalSemigroup(L);; 
Betti := 0;;

for lam in [(2*e+2)..(4*e-2)] do
ZD := []; OD := [];

for i in [1..Length(L)] do 
  if (lam - L[i] in S) then Append(ZD,[L[i]]); fi;
  for j in [(i+1)..Length(L)] do
    if (lam - L[i] - L[j] in S) then Append(OD,[[L[i],L[j]]]); fi;
  od;
od;  
Print("----------\nlambda = ", lam, "\n"); 
Print("0-dim faces: ", ZD, "\n1-dim faces: ", OD, "\n");

if Length(OD) > 0 then
  Mat := NullMat(Length(OD),Length(ZD));
  for i in OD do
    Mat[Position(OD,i)][Position(ZD,i[1])] := 1;
    Mat[Position(OD,i)][Position(ZD,i[2])] := -1;
  od;
  Rk := RankMat(Mat);
  Print("Number of 0-dim faces: ", Length(ZD), "\nRank M: ", Rk, "\n");
  GB := Length(ZD)-1-Rk;
  Print("Graded Betti no: ", GB, "\n");
  Betti := Betti + GB;
else
  Print("Graded Betti no: 0\n");
fi;
od;
Print("----------\n");
Print("Minimal number of generators of the defining equation of S(", e, ",", m, ",", n,") 
is: ", Betti);
end;
\end{verbatim}
For example, for the semigroup $S_10(2,6)$, the command \texttt{GradedBetti(10,2,6)} can be used to compute the graded Betti numbers and the first Betti number.

\section{Appendix}

\subsubsection{Computing $\b$ when $\l =2e+2m$}\label{x=2m-5}

For $2 \leq m \leq \left\lfloor \frac{e-1}{2} \right\rfloor$ and $j<l \in A \setminus \{2m-n\},$ $$\{e+j,e+l\} \in \D_\l \Leftrightarrow \l - (e+j)-(e+l)=0 \Leftrightarrow j+l = 2m.$$ This gives us the following pairs for $(j,l)$: 
     \begin{center}
\begin{tikzpicture}
    \node at (4,3) {$0 < 1 < \cdots < m-1 < m+1 < \cdots < 2m-1 < 2m$};
    \draw (0.25,3.25) -- (0.25,3.75);
    \draw (0.25,3.75) -- (7.7,3.75);
    \draw (7.7,3.75) -- (7.7, 3.25);
    \draw (3,3.25) -- (3,3.50);
   \draw (3,3.50) -- (4,3.50);
    \draw (4,3.50) -- (4,3.25);
    \draw (0.95,2.75) -- (0.95,2.5);
    \draw (0.95,2.5) -- (6.85,2.5);
    \draw (6.85,2.5) -- (6.85, 2.75);
\end{tikzpicture}
\end{center}
When $m \leq \ll \frac{n-1}{2} \rr$, $n$ is not in this sequence but $m$ is. After removing the latter we get $m$ pairs. When $\lc \frac{n}{2} \rc \leq m \leq \left\lfloor \frac{e-1}{2} \right\rfloor$, both $n$ and $2m-n$ are in the above sequence along with $m$. After removing all three, we get $m-1$ pairs.

When $ \lc \frac{e}{2} \rc < m \leq e-3 $, we get the one-dimensional faces $\{e,e+j\}$ for all $j \in [1,2m-e]$. Note that $n > m > 2m-e$ since $m <e$, and $2m-n > 2m-e$ so none of these elements are in the set $[1,2m-e].$ Thus, we get $2m-e$ one-dimensional faces from here. 

Any other $1$-dimensional face that adds to the rank of $M$ must contain $e+l$ for some $l \geq 2m-e+1$. This happens only when $\l -(e+j)-(e+l) =2m-j-l=0 \in \SM{m,n}$. So we have the following pairs $(j,l)$ with $j+l=2m$.

\begin{center}
\begin{tikzpicture}
    \node at (4,3) {$2m-e+1 < 2m-e+2 < \cdots < m-1 < m+1 < \cdots < e-2 < e-1.$};
    \draw (-0.75,3.25) -- (-0.75,3.75);
    \draw (-0.75,3.75) -- (8.85,3.75);
    \draw (8.85,3.75) -- (8.85, 3.25);
    \draw (4.25,3.25) -- (4.25,3.50);
   \draw (4.25,3.50) -- (5.5,3.50);
    \draw (5.5,3.50) -- (5.5,3.25);
    \draw (1.75,2.75) -- (1.75,2.5);
    \draw (1.75,2.5) -- (7.75,2.5);
    \draw (7.75,2.5) -- (7.75, 2.75);
\end{tikzpicture}
\end{center}

This gives us $e-m-2$ pairs after taking away $(n, 2m-n)$.

When $m = \lc e/2 \rc$, then we don't get any faces of the form $\{e, e+j\}$, and the only one dimensional faces $\{e+j,e+l\}$ occur when $2m-j-l=0$ giving the following $m-2$ pairs after taking away $m,n,2m-n$ from the sequence:  
\begin{center}
\begin{tikzpicture}
    \node at (4,3) {$1 < 2 < \cdots < 2m-2 < 2m-1$};
    \draw (1.5,3.25) -- (1.5,3.50);
    \draw (1.5,3.50) -- (6,3.50);
    \draw (6,3.50) -- (6,3.25);
    \draw (2.2,2.75) -- (2.2,2.5);
    \draw (2.2,2.5) -- (4.75,2.5);
    \draw (4.75,2.5) -- (4.75, 2.75);
\end{tikzpicture}
\end{center}

\subsubsection{Computing $\b$ when $\l = 2e+m+n$}\label{x=m+n-5}

When $x=m+n-5$, the computations are almost the same. When $2 \leq m \leq e-n-1$, the only one-dimensional faces are $\{e+j,e+l\}$ whenever $j<l \in A$ with $j+l=m+n$, i.e. the following: \begin{center}
\begin{tikzpicture}
    \node at (4,3) {$0 < 1 < \cdots < m < \cdots < n < \cdots < m+n-1 < m+n$};
    \draw (-0.35,3.25) -- (-0.35,3.75);
    \draw (-0.35,3.75) -- (8,3.75);
    \draw (8,3.75) -- (8, 3.25);
    \draw (2,3.25) -- (2,3.50);
   \draw (2,3.50) -- (3.7,3.50);
    \draw (3.7,3.50) -- (3.7,3.25);
    \draw (0.35,2.75) -- (0.35,2.5);
    \draw (0.35,2.5) -- (6.5,2.5);
    \draw (6.5,2.5) -- (6.5, 2.75);
\end{tikzpicture}
\end{center}
Excluding the pair $\{m,n\}$, we get $\ll \frac{m+n-1}{2} \rr$ one-dimensional faces. Thus, $\im \d_1 \cong \Z^{\ll \frac{m+n-1}{2} \rr}$.

When $e-n \leq m \leq e-3$, we have the one-dimensional faces $\{e,e+j\}$ for all $j \in [1,m+n-e]$. Note that $ m,n \not\in [1,m+n-e]$ as $m = m+n-n > m+n-e$. This gives us $m+n-e$ faces. We also have the $\ll \frac{2e-m-n-3}{2} \rr$ faces given by the following pairs, excluding $\{m,n\}$: 
\begin{center}
\begin{tikzpicture}
    \node at (2.1,3) {$m+n-e+1 <  \cdots < m < \cdots < n < \cdots < e-1$};
    \draw (2,3.25) -- (2,3.6);
   \draw (2,3.6) -- (3.7,3.6);
    \draw (3.7,3.6) -- (3.7,3.25);
    \draw (-1,2.75) -- (-1,2.4);
    \draw (-1,2.4) -- (5.75,2.4);
    \draw (5.75,2.4) -- (5.75, 2.75);
\end{tikzpicture}
\end{center}
Thus, $\rk M = m+n-e + \ll \frac{2e-m-n-3}{2} \rr = \ll \frac{m+n-3}{2} \rr$.

 \medskip

\subsubsection{Computing $\b$ when $\l=2e+2n$}\label{x=2n-5}

When $3 \leq n \leq \ll \frac{e-1}{2} \rr,$ the only one-dimensional faces are $\{e+j,e+l\}$ with $j+l=2n$ as $\l - (e+j)-(e+l) \leq e-2$ for any $j<l \in A\setminus B$. These are the following faces:
 \begin{center}
\begin{tikzpicture}
    \node at (4,3) {$0 < 1 < \cdots < n-1 < n+1 < \cdots < 2n-1 < 2n$};
    \draw (0.25,3.25) -- (0.25,3.75);
    \draw (0.25,3.75) -- (7.7,3.75);
    \draw (7.7,3.75) -- (7.7, 3.25);
    \draw (3,3.25) -- (3,3.50);
   \draw (3,3.50) -- (4,3.50);
    \draw (4,3.50) -- (4,3.25);
    \draw (0.95,2.75) -- (0.95,2.5);
    \draw (0.95,2.5) -- (6.85,2.5);
    \draw (6.85,2.5) -- (6.85, 2.75);
\end{tikzpicture}
\end{center}
From this, we take away $m,2n-m,$ and $n$ to get $\im \d_1 \cong \Z^{n-1}$ in this case.

When $\lc \frac{e}{2} \rc \leq n \leq \ll \frac{e+m-1}{2} \rr,$ we have the 1-dimensional faces $\{e,e+j\}$ for all $j \in [1,2n-e].$ Further, as $m, n $ and $2n-m$ are not in this set, we get $2n-e$ faces from here. Note that when $n= \lc e/2 \rc$, no 1-dimensional faces of this form are possible but then $2n-e=0$ still gives the right number.  All other faces that add to $\rk M$ are of the form $\{e+j,e+l\}$ with $j+l = 2n$ and $j \geq 2n-e+1$. These are the following $e-n-2$ faces, after taking away $m,n,2n-m$ as all three are in the sequence:

\begin{align} \label{app:3:2}
\begin{tikzpicture} 
    \node at (4,3) {$2n-e+1 < 2n-e+2 < \cdots < n-1 < n+1 < \cdots < e-2 < e-1.$};
    \draw (-0.75,3.25) -- (-0.75,3.75);
    \draw (-0.75,3.75) -- (8.85,3.75);
    \draw (8.85,3.75) -- (8.85, 3.25);
    \draw (4.25,3.25) -- (4.25,3.50);
   \draw (4.25,3.50) -- (5.5,3.50);
    \draw (5.5,3.50) -- (5.5,3.25);
    \draw (1.75,2.75) -- (1.75,2.5);
    \draw (1.75,2.5) -- (7.75,2.5);
    \draw (7.75,2.5) -- (7.75, 2.75);
\end{tikzpicture}
\end{align}
So, in this case, we get $\im \d_1 \cong \Z^{n-2}$. 

When $\lc \frac{e+m}{2} \rc \leq n \leq e-2,$ we also have the $1$-dim faces $\{e,e+j\}$ for all $j \in [1,2n-e]$. As $n$ is not in this set but $m$ is, we get $2n-e-1$ faces from here. The only other one-dimensional faces that add to rank $M$ are as in \eqref{app:3:2}.
As both $m, 2n-m$ are not in this sequence but $n$ is, this gives us $e-n-1$ faces. So, in this case, we get $\im \d_1 \cong \Z^{n-2}$. 

\subsubsection{Adding $\b$ from Part 2}\label{sum-betti-part2}

We begin by adding the values from the table \ref{mainBox} assuming that $m \neq 2$ and $(m,n) \neq (3,4).$

\begin{align} \label{even-odd}
    \sum_{\l = 2e+2}^{4e-2} \b =& \sum_{\l=2e+2}^{3e-1} \left\lfloor \frac{\l}{2} \right\rfloor - e(e-2) - (n-m) - 2(e-n) + (2e-3)(e-1) \\ \nonumber
    &- \sum_{\l=3e}^{4e-2} \left\lfloor \frac{\l-1}{2} \right\rfloor + (n-m) + 2(e-n-1) \\ \nonumber
    =& \ e^2-3e+1 + \sum_{\l=2e+2}^{3e-2} \left\lfloor \frac{\l}{2} \right\rfloor - \sum_{\l=3e}^{4e-3} \left\lfloor \frac{\l}{2} \right\rfloor. 
\end{align} 
When $e$ is even, \eqref{even-odd} simplifies to the following:
\begin{align*}
    \sum_{\l = 2e+2}^{4e-2} \b =& \
    e^2-3e+1 +2\sum_{i=1}^{e/2-2} (e+i) + \frac{3e}{2}-1 - 2\sum_{i=0}^{e/2-2} \left( \frac{3e}{2} +i \right) \\
    =& \ e^2 - \frac{3e}{2} + 2e \left( \frac{e}{2}-2 \right) - 3e \left( \frac{e}{2} -1 \right) 
     = \ \frac{e^2}{2} - \frac{5e}{2}.
\end{align*}

When $e$ is odd, \eqref{even-odd} simplifies to the following:
\begin{align*}
    \sum_{\l = 2e+2}^{4e-2} \b =& \
    e^2-3e+1 +2\sum_{i=1}^{(e-3)/2} (e+i) - \sum_{i=0}^{(e-3)/2} \left( \frac{3e-1}{2}+i \right) - \sum_{i=0}^{(e-5)/2} \left( \frac{3e+1}{2} +i \right) \\
    =& \ e^2 - 3e + 1 + e(e-3) + \frac{1}{4}(e-3)(e-1) - \frac{1}{4} (3e-1)(e-1) - \frac{1}{8} (e-3)(e-1) \\
    &- \frac{1}{4} (3e+1)(e-3) - \frac{1}{8} (e-5)(e-3)
     = \ \frac{e^2}{2} - \frac{5e}{2}.
\end{align*}

\vspace{0.5cm}
\begin{acknowledgment*}{\rm
This work was supported by the Association for Women in Mathematics under Grant No. NSF-DMS 2113506, which partially funded the Women in Commutative Algebra III Workshop. The workshop was hosted and financially supported by Banff International Research Station at its CMO Oaxaca location. Additional financial support was provided by the NSF grant DMS–2433082. Hema Srinivasan is supported by grants from Simons Foundation. Kriti Goel was also supported by grant CEX2021-001142-S (Excelencia Severo Ochoa) funded by MICIU/AEI/10.13039/501100011033 (Ministerio de Ciencia, Innovaci\'on y Universidades, Spain).
We are also grateful to the software systems Singular \cite{Singular} and GAP \cite{NumericalSgps} for serving as an excellent source of inspiration.
}\end{acknowledgment*}

\end{document}